\documentclass[10pt]{amsart}
\usepackage{setspace}

\usepackage{hyperref}

\usepackage{todonotes}

\usepackage{mathtools}
\usepackage{amsmath}
\usepackage{amssymb}
\usepackage{amsthm}

\usepackage{color}
\usepackage[all]{xy}
\usepackage{verbatim}
\usepackage{mathrsfs}
\usepackage{graphicx}
\usepackage{tikz-cd}

\usepackage{url}
\usepackage{verbatim}

\usepackage{xy}
\xyoption{all}

\newtheorem{lemma}{Lemma}[section]

\newtheorem{theorem}[lemma]{Theorem}

\newtheorem{prop}[lemma]{Proposition}

\theoremstyle{definition}

\newtheorem{definition}[lemma]{Definition}
\newtheorem{example}[lemma]{Example}

\newtheorem{remark}[lemma]{Remark}

\newcommand{\Z}{\mathbf{Z}}
\newcommand{\QQ}{\mathbf{Q}}

\newcommand{\wt}[1]{\widetilde{#1}}

\newcommand{\co}{\mathscr{O}}

\newcommand{\GG}{\mathbb{G}}

\newcommand{\cA}{\mathscr{A}}

\renewcommand{\H}{\widehat{\mathscr{H}}}

\newcommand{\cris}{\mathrm{cris}}

\newcommand{\fr}[1]{\mathfrak{#1}}

\newcommand{\ol}[1]{\overline{#1}}

\renewcommand{\sp}{\mathrm{sp}}
\newcommand{\et}{{\acute{\mathrm{e}}\mathrm{t}}}
\renewcommand{\H}{{H}}

\newcommand{\gal}[1]{G_{#1}}

\newtheorem{abc}{abc}
\newtheorem{letterthm}[abc]{Theorem}

\title{Galois representations and ordinary reduction}
\author{Sanath Devalapurkar}
\begin{document}
\begin{abstract}
    We provide conditions on the $p$-adic Galois representation of a smooth
    proper variety over a complete nonarchimedean extension of $\QQ_p$ to have
    (potentially) good ordinary reduction.
\end{abstract}

\maketitle

\section{Introduction}
Let $K/\QQ_p$ be a finite extension with perfect residue field $k$ of
characteristic $p$, and let $C$ be a completion of an algebraic closure of $K$.
A smooth proper variety $X_0/K$ has \emph{good reduction} if there it admits a
smooth proper model $\fr{X}_0$ over $\co_K$. Let $X = X_0\otimes_K\ol{K}$. In
this case:
\begin{enumerate}
    \item The $\gal{K}$-representation on $\H^n_\et(X;\QQ_\ell)$ is unramified
	for all $n$ and all primes $\ell\neq p$; moreover,
    \item The $\gal{K}$-representation $\H^n_\et(X;\QQ_p)$ is crystalline for
	all $n$.
\end{enumerate}

In many cases, there is a converse to this result. For instance, when $X_0$ is
an abelian variety, the N\'eron--Ogg--Shafarevich theorem (see
\cite{serre-tate}) provides a converse: $H^1_\et(X;\QQ_\ell)$ is an unramified
$\gal{K}$-representation for one (hence all) primes $\ell\neq p$ if and only if
$X_0$ has good reduction. The $p$-adic analogue of this result was proved by
Coleman--Iovita (\cite[Theorem II.4.7]{coleman-iovita}): $X_0$ has good
reduction if and only if $H^1_\et(X;\QQ_p)$ is a crystalline $p$-adic
$\gal{K}$-representation. In \cite{neron-ogg-k3}, Chiarellotto, Lazda, and
Liedtke obtained analogous statements for certain K3 surfaces.

Let $A$ be an abelian variety over $K$ of dimension $g$ with good reduction, so
there exists a smooth proper model $\cA$ over $\co_K$ whose generic fiber is
$K$. The special fiber $\cA_\sp$ is an abelian variety over a perfect field of
characteristic $p$. The associated $p$-divisible group $\cA_\sp[p^\infty]$ has
height $2g$; however, its connected component $\cA_\sp[p^\infty]^\circ$ need
not have the same height. The height of the formal group
$\cA_\sp[p^\infty]^\circ$ provides a stratification of the moduli space of
abelian varieties. The subset with maximal height (namely, $g$) is dense in
this moduli space; in this case, the abelian variety has ordinary reduction.

In light of the above discussion, it is natural to ask if there is a
Galois-theoretic criterion on the \'etale cohomology of $A\otimes_K\ol{K}$
which ensures that $A$ has ordinary reduction. The answer is yes:
\begin{letterthm}
    \label{thm1}
    Let $A$ be an abelian variety with good reduction, and let $V$ denote the
    $p$-adic $\gal{K}$-representation $\H^n_\et(A\otimes_K \ol{K};\QQ_p)$. Then
    $A$ is ordinary iff there is a complete filtration $F^i V$ of $V$ which
    splits, such that the inertia subgroup of $\gal{K}$ acts trivially on $F^i
    V/F^{i+1} V$ for $i>\dim A$, and by the $p$-adic cyclotomic character for
    $i\leq \dim A$.
\end{letterthm}
\newtheorem*{introthm1}{Theorem \ref{thm1}}
\begin{remark}
    This result is probably well-known to the experts; however, we were unable
    to find a reference in the literature, so we include a proof in this paper.
\end{remark}
Note that as ordinarity is a condition on the $p$-divisible group of $\cA_\sp$,
one only expects such a criterion on the $p$-adic \'etale cohomology of
$A\otimes_K\ol{K}$, and not its $\ell$-adic \'etale cohomology. 

Motivated by this result, one can ask if an analogue of Theorem \ref{thm1}
holds for smooth proper varieties which are not necessarily abelian varieties.
In order to do this, one of course needs an appropriate definition of the term
``ordinary'' for general smooth proper varieties. This was provided by Bloch
and Kato in \cite{bloch-kato}. Let $X$ be a smooth proper variety over $K$ with
good reduction, and let $\fr{X}$ denote a smooth proper model for $X$ over
$\co_K$.
\begin{definition}
    Let $d\Omega^j_{\fr{X}_\sp}$ denote the sheaf of exact differentials on
    $\fr{X}_\sp$. We say that $X$ has \emph{ordinary reduction} if
    $H^i(\fr{X}_\sp, d\Omega^j_{\fr{X}_\sp})$ is $0$ for all $i,j$.
\end{definition}

In analogy with the condition imposed in Theorem \ref{thm1}, we also make the
following definition.
\begin{definition}\label{ordinary-gal-rep}
    A $\gal{K}$-representation $V$ is \emph{ordinary} if there is a finite
    filtration by $\gal{K}$-stable vector spaces $F^i V$ such that the inertia
    subgroup of $\gal{K}$ acts on $F^i V/F^{i+1} V$ by some power $\chi^{n_i}$
    of the cyclotomic character.
\end{definition}

The main goal of this paper is to prove the following analogues of Theorem
\ref{thm1}:
\begin{letterthm}
    \label{thm2}
    Let $X_0$ be a smooth proper variety over $K$ with a smooth proper model
    $\fr{X}$ over $\co_K$ such that $H^\ast_\cris(\fr{X}_\sp/W(k))$ and
    $H^\ast(\fr{X}_\sp,W\Omega^\ast_{\fr{X}_\sp})$ are torsion-free. The
    \'etale cohomology $H^i_\et(X;\QQ_p)$ is an ordinary
    $\gal{K}$-representation whose associated filtration (from Definition
    \ref{ordinary-gal-rep}) splits for all $i$ if and only if $X_0$ has
    ordinary reduction.
\end{letterthm}
\newtheorem*{introthm2}{Theorem \ref{thm2}}
\begin{letterthm}
    \label{thm3}
    Let $X_0$ be a K3 surface over $K$ with good reduction. The \'etale
    cohomology $H^2_\et(X;\QQ_p)$ is an ordinary $\gal{K}$-representation whose
    associated filtration (from Definition \ref{ordinary-gal-rep}) splits if
    and only if $X_0$ has ordinary reduction.
\end{letterthm}
\newtheorem*{introthm3}{Theorem \ref{thm3}}
\begin{remark}
    Combined with \cite{matsumoto-k3}, we obtain conditions on the
    $\gal{K}$-representation associated to a K3 surface guaranteeing
    potentially good ordinary reduction.
\end{remark}
We will deduce these results from the following more general theorem.
\begin{letterthm}
    \label{thm4}
    Suppose $X_0$ is a smooth proper variety over $K$ with good reduction such
    that the special fiber of some model $\fr{X}$ of $X_0$ over $\co_K$ has
    torsion-free crystalline cohomology. If $H^n_\et(X;\QQ_p)$ is an ordinary
    $\gal{K}$-representation whose filtration (from Definition
    \ref{ordinary-gal-rep}) splits, then the Newton polygon for $\fr{X}_\sp$
    has integer slopes.
\end{letterthm}
\newtheorem*{introthm4}{Theorem \ref{thm4}}

\subsection*{Acknowledgements}
Thanks to Matthew Emerton, Mark Kisin, Sameera Vemulapalli, and David
Zureick-Brown for chatting with me. Thanks also to the anonymous reader who
pointed out errors in an earlier draft of this writeup, and to Bjorn Poonen for
forwarding these comments to me.

\section{Background}
\subsection{Proving Theorem \ref{thm1}}
Our goal is to prove Theorem \ref{thm1}, which we will recall here for the
reader's convenience.
\begin{introthm1}
    Let $A$ be an abelian variety with good reduction, and let $V$ denote the
    $p$-adic $\gal{K}$-representation $\H^n_\et(A\otimes_K \ol{K};\QQ_p)$. Then
    $A$ is ordinary iff there is a complete filtration $F^i V$ of $V$ which
    splits, such that the inertia subgroup of $\gal{K}$ acts trivially on $F^i
    V/F^{i+1} V$ for $i>\dim A$, and by the $p$-adic cyclotomic character for
    $i\leq \dim A$.
\end{introthm1}
\begin{proof}
    In this case, we only have to consider the case $n=1$. Assume that $A$ is
    ordinary. Since $A$ has good reduction, we lift it to a smooth proper model
    $\cA$ over $\co_K$. Let $\cA[p^\infty]$ denote the $p$-divisible group of
    its special fiber. The connected-\'etale sequence runs
    $$0\to \cA[p^\infty]^\circ \to \cA[p^\infty]\to \cA[p^\infty]^\et\to 0;$$
    taking the rational $p$-adic Tate module gives
    $$0\to V_p(\cA[p^\infty]^\circ) \to V_p(\cA[p^\infty]) \to
    V_p(\cA[p^\infty]^\et) \to 0.$$
    The $\gal{K}$-representation $V_p(\cA[p^\infty]^\et)$ is unramified, while
    the inertia subgroup of $\gal{K}$ acts on $V_p(\cA[p^\infty]^\circ)$ by a
    direct sum of copies of the $p$-adic cyclotomic character. By \cite[Theorem
    2.4]{neron-ogg-k3}, we know that as $\gal{K}$-representations, we have
    $$H^n_\et(A\otimes_K\ol{K}; \QQ_p) \simeq
    H^n_\et(\cA\otimes_{\co_K}\widehat{\ol{K}}; \QQ_p).$$
    Note that $V_p(A) = H^1_\et(A\otimes_K\ol{K};\QQ_p)^\vee$, so
    $V_p(\cA[p^\infty])$ can be identified with $V_p(A) =: V$ as
    $\gal{K}$-representations. As $A$ is ordinary, we know that $\cA[p^\infty]$
    is an extension of an \'etale $p$-divisible group by a multiplicative
    $p$-divisible group; taking the rational Tate module of the resulting
    filtration on $\cA[p^\infty]$ gives a filtration on $V$ (with $F^{\dim A} V
    = V_p(\cA[p^\infty]^\circ)$) satisfying the conditions listed in Theorem
    \ref{thm1}.

    For the converse, suppose that the $p$-adic $\gal{K}$-representation
    $\H^1_\et(A\otimes_K\ol{K};\QQ_p) =: V$ satisfies the conditions in Theorem
    \ref{thm1}. Then $V$ contains a $\gal{K}$-stable $\Z_p$-lattice
    $\Lambda:=H^1_\et(A\otimes_K\ol{K};\Z_p)$.  Under the correspondence of
    \cite[Proposition 7.2.2]{brinon-conrad}, this lattice comes from the
    $p$-divisible group $A[p^\infty]$. Since $A$ has good reduction, this can
    be lifted to a $p$-divisible group over $\co_K$.

    Let $W\subsetneq V$ denote the $\dim A$-th step of the filtration on $V$.
    We get a $\gal{K}$-stable $\Z_p$-lattice $W \cap \Lambda$ of $W$ which is
    contained inside $\Lambda$. Again referring to \cite[Proposition
    7.2.2]{brinon-conrad}, we can find a $p$-divisible subgroup $\GG$ of
    $A[p^\infty]$. This in turn can be lifted to a subgroup $\wt{\GG}$ of the
    $p$-divisible group over $\co_K$. By our assumptions on $V$, the associated
    $\gal{K}$-representation itself admits a complete split filtration, with
    the inertia subgroup of $\gal{K}$ acting on each quotient by the $p$-adic
    cyclotomic character. Thereby identifying $\wt{\GG}$ with a multiplicative
    $p$-divisible group, we find that the special fiber $\cA_\sp$ is an
    ordinary abelian variety, as desired.
\end{proof}

\subsection{Breuil--Kisin modules}
In order to prove Theorem \ref{thm4}, we will need to recall some of the theory
of Breuil--Kisin modules. Let $\fr{S} = W(k)[[T]]$. There is a Frobenius
$\varphi$ on $\fr{S}$ defined by the usual Frobenius on $W(k)$ and the map
$T\mapsto T^p$. The map $\fr{S}\to \co_K$ sending $T$ to a uniformizer $\pi$
has kernel generated by $E(T)$, the Eisenstein polynomial for $\pi$. Recall:
\begin{definition}
    A Breuil--Kisin module $M$ is a $\fr{S}$-module along with an isomorphism
    $$\varphi_M: M\otimes_\fr{S}^\varphi \fr{S}[\frac{1}{E}] \simeq
    M[\frac{1}{E}].$$
\end{definition}
\begin{example}[Tate twists]
    Let $\fr{S}\{1\}$ denote the Breuil--Kisin module with underlying
    $\fr{S}$-module $\fr{S}$ and Frobenius $\varphi_{\fr{S}\{1\}}(x) =
    \frac{u}{E(T)}\varphi(x)$, with $u$ some explicit unit in $\fr{S}$. We
    write $\fr{S}\{n\} = \fr{S}\{1\}^{\otimes n}$.
\end{example}
\begin{remark}
    In analogy with Definition \ref{ordinary-gal-rep}, we say that a
    Breuil--Kisin module $(M,\varphi_M)$ is ordinary if there is a filtration
    by submodules $(F^i M,\varphi_M|_{F^i M})$ such that each successive
    quotient is a finite free Breuil--Kisin module of rank $1$.
\end{remark}
We will need the following result (see \cite[Theorem 4.4]{bms-i}):
\begin{theorem}\label{bk}
    There is a fully faithful tensor functor $M$ from $\Z_p$-lattices $\Lambda$
    in crystalline $\gal{K}$-representations $V$ to finite free Breuil--Kisin
    modules, characterized by the property that there is a
    $\varphi,\gal{K_\infty}$-equivariant identification
    $$M(\Lambda)\otimes_\fr{S} W(C^\flat) \simeq \Lambda\otimes_{\Z_p}
    W(C^\flat).$$
\end{theorem}
\begin{example}\label{tate}
    We have $M(\Z_p(n)) \simeq \fr{S}\{n\}$.
\end{example}
The functor of Theorem \ref{bk} is not generally exact. However, we have:
\begin{lemma}\label{ordinary-bk}
    Let $V$ be an ordinary crystalline $\gal{K}$-representation, and let
    $\Lambda$ be
    a $\gal{K}$-stable $\Z_p$-lattice in $V$. Then $M(\Lambda)$ is an ordinary
    Breuil--Kisin module.
\end{lemma}
\begin{proof}
    Let $F^i V$ be the filtration on $V$ (from Definition
    \ref{ordinary-gal-rep}). Then $F^i V \cap \Lambda =: F^i \Lambda$ forms a
    filtration of $\gal{K}$-stable $\Z_p$-lattices inside $\Lambda$. As
    $W(C^\flat)$ is torsion-free, the characterization of $M(\Lambda)$ from
    Theorem \ref{bk} proves that the rank of $M(\Lambda)$ as a $\fr{S}$-module
    is the rank of $\Lambda$ as a $\Z_p$-lattice. It follows that $M(F^i
    \Lambda)/M(F^{i+1} \Lambda)$ is a finite free Breuil--Kisin module of rank
    $1$.
\end{proof}
\begin{remark}\label{split}
    If the filtration on $V$ splits, then so does the filtration on $M(\Lambda)$.
    Indeed, there is a canonical map $M(F^i \Lambda)/M(F^{i+1} \Lambda)\to
    M(F^i \Lambda/F^{i+1} \Lambda)$; by functoriality, we get a map $M(F^i
    \Lambda/F^{i+1} \Lambda)$ which splits the exact sequence defining $M(F^i
    \Lambda)/M(F^{i+1} \Lambda)$.
\end{remark}

\section{The proof of Theorem \ref{thm4}}
We recall Theorem \ref{thm4} for the reader's convenience.
\begin{introthm4}
    Suppose $X_0$ is a smooth proper variety over $K$ with good reduction such
    that the special fiber of some model $\fr{X}$ of $X_0$ over $\co_K$ has
    torsion-free crystalline cohomology. If $H^n_\et(X;\QQ_p)$ is an ordinary
    $\gal{K}$-representation whose filtration splits, then the Newton polygon
    for $\fr{X}_\sp$ has integer slopes.
\end{introthm4}
\begin{proof}
    Suppose that $X_0$ is as above, and that $V = H^n_\et(X;\QQ_p)$. Since
    $H^\ast_\cris(\fr{X}_\sp/W(k))$ is torsion-free, we learn from
    \cite[Theorem 1.1(ii)]{bms-i} and \cite[Theorem 2.4]{neron-ogg-k3} that
    $H^n_\et(X;\Z_p) =: T$ is a $\gal{K}$-stable $\Z_p$-lattice inside $V$. The
    discussion preceding \cite[Theorem 1.4]{bms-i} implies that
    \begin{equation}\label{cris-compare}
	M(T)\otimes_{\fr{S}} W(k) = H^n_\cris(\fr{X}_\sp/W(k))
    \end{equation}
    Suppose $V$ is an ordinary $\gal{K}$-representation whose filtration
    splits. Then by Lemma \ref{ordinary-bk}, Remark \ref{split}, and Equation
    \eqref{cris-compare}, we find that $H^n_\cris(\fr{X}_\sp/W(k))$ is a
    $F$-crystal over $k$ which splits into a direct sum of rank one
    $F$-crystals.  By the Dieudonn\'e--Manin classification, rank one
    $F$-crystals are of the form $M_{r/1} = W(k)\langle T\rangle/(T = p^r)$ for
    $r\in \Z_{\geq 0}$. In particular, the Newton polygon for $\fr{X}_\sp$ has
    integer slopes, as desired.
\end{proof}
\begin{remark}
    The slopes of the Hodge polygon of a smooth proper variety are always
    integers.
\end{remark}
\begin{proof}[Proof of Theorem \ref{thm2}]
    This follows from Theorem \ref{thm4} and \cite[Proposition
    7.3(7)]{bloch-kato}.
\end{proof}

In order to prove Theorem \ref{thm3}, we will need the
following result from \cite{katz-products}:
\begin{prop}\label{ordinary-cond}
    Let $X_0$ be a smooth proper variety over $K$ with good reduction. Let
    $\fr{X}$ be a lift of $X_0$ to $\co_K$. Suppose that
    $H^\ast_\cris(\fr{X}_\sp/W(k))$ is torsion-free. Then $X_0$ is ordinary if
    and only if the following conditions are satisfied:
    \begin{itemize}
	\item the Hodge--de-Rham spectral sequence
	    $$E_1^{s,t} = H^t(\fr{X}_\sp;\Omega^s_{\fr{X}_\sp/k}) \Rightarrow
	    H^{s+t}_\mathrm{dR}(\fr{X}_\sp/k)$$
	    collapses at the $E_1$-page.
	\item the Newton and Hodge polygons of $H^\ast_\cris(\fr{X}_\sp/W(k))$
	    coincide.
    \end{itemize}
\end{prop}

\begin{proof}[Proof of Theorem \ref{thm3}]
    It is a general fact about K3 surfaces that $H^\ast_\cris(\fr{X}_\sp/W(k))$
    is torsion-free. By \cite[Proposition 2.5]{liedtke}, the conditions of
    Theorem \ref{thm4} are satisfied for every K3 surface having good
    reduction. By Theorem \ref{thm4}, we find that the slopes of the Newton
    polygon of $X_0$ must be integers. Moreover, our hypotheses on $X_0$
    imply that it cannot be a supersingular K3 surface. By the Artin--Mazur
    classification of the heights of K3 surfaces, we conclude that the slopes
    of the Newton polygon of $X_0$ can be integers if and only if the height of
    $X_0$ is $1$; this implies that the Hodge and Newton polygons coincide. By
    Proposition \ref{ordinary-cond}, we conclude that $X_0$ has potentially
    good ordinary reduction, as desired.
\end{proof}
\begin{remark}
    The same argument gives another proof of Theorem \ref{thm1}.
\end{remark}
\bibliographystyle{alpha}
\bibliography{../../main}
\end{document}